\documentclass[coll]{imsart}

\usepackage{amsthm,amsmath,amsfonts,amsbsy,amssymb,graphicx,placeins}
\usepackage{url}

\startlocaldefs

\usepackage[sort,longnamesfirst]{natbib}
\newcommand{\pcite}[1]{\citeauthor{#1}'s \citeyearpar{#1}}

\newtheorem{theorem}{Theorem}
\newtheorem{corollary}{Corollary}
\newtheorem{lemma}{Lemma}

\theoremstyle{remark}

\newtheorem{definition}{Definition}

\newfont{\msbm}{msbm10 at 11pt}

\endlocaldefs
\begin{document}

\begin{frontmatter}

\title{On Convergence Properties of the Monte Carlo EM Algorithm}
\runtitle{Convergence of MCEM}

\author{\fnms{Ronald C.} \snm{Neath}\corref{}\ead[label=e1]{rneath@hunter.cuny.edu}}
\address{Department of Mathematics and Statistics 
Hunter College, City University of New York
\printead{e1}}
\affiliation{Hunter College, City University of New York}

\begin{abstract}
  The Expectation-Maximization (EM) algorithm
  \citep{demp:lair:rubi:1977} is a popular method for computing
  maximum likelihood estimates (MLEs) in problems with missing data.
  Each iteration of the algorithm formally consists of an E-step:
  evaluate the expected complete-data log-likelihood given the
  observed data, with expectation taken at current parameter estimate;
  and an M-step: maximize the resulting expression to find the updated
  estimate.  Conditions that guarantee convergence of the EM sequence
  to a unique MLE were found by \citet{boyl:1983} and \citet{wu:1983}.
  In complicated models for high-dimensional data, it is common to
  encounter an intractable integral in the E-step.  The Monte Carlo EM
  algorithm of \citet{wei:tann:1990} works around this difficulty by
  maximizing instead a Monte Carlo approximation to the appropriate
  conditional expectation.  Convergence properties of Monte Carlo EM
  have been studied, most notably, by \citet{chan:ledo:1995} and
  \citet{fort:moul:2003a}.

  The goal of this review paper is to provide an accessible but
  rigorous introduction to the convergence properties of EM and Monte
  Carlo EM.  No previous knowledge of the EM algorithm is assumed.  We
  demonstrate the implementation of EM and Monte Carlo EM in two
  simple but realistic examples.  We show that if the EM algorithm
  converges it converges to a stationary point of the likelihood, and
  that the rate of convergence is linear at best.  For Monte Carlo EM
  we present a readable proof of the main result of
  \citet{chan:ledo:1995}, and state without proof the conclusions of
  \citet{fort:moul:2003a}.  An important practical implication of
  \pcite{fort:moul:2003a} result relates to the determination of Monte
  Carlo sample sizes in MCEM; we provide a brief review of the
  literature \citep{boot:hobe:1999, caff:jank:jone:2005} on that
  problem.
\end{abstract}



\end{frontmatter}

\section{Introduction: The Monte Carlo EM algorithm}
\label{sec:Intro}

The expectation-maximization, or EM algorithm, is an algorithm for 
maximizing likelihood functions, especially in the presence of 
missing data.  When EM works, the algorithm's output is a sequence of 
parameter values that converges to the maximum likelihood estimate (MLE).  
The seminal paper on EM, and that which gave the algorithm its name, is 
the article by \citet{demp:lair:rubi:1977}.  A book length treatment is 
given by \citet{mcla:kris:1997}.  

Consider a statistical model in which the random vector $(Y, U)$, 
$Y \in \mathbb{R}^N$ and $U \in \mathbb{R}^q$, has distribution given 
by $f(y, u; \theta)$, a density with respect to the measure $\lambda 
\times \mu$, where $\lambda$ and $\mu$ are measures on $\mathbb{R}^N$ and 
$\mathbb{R}^q$ respectively, and indexed by the unknown parameter 
$\theta \in \Theta$.  We refer to $(Y,U)$ as the ``complete data'' but 
only $Y = y$ is observed; $U$ represents the unobserved or ``missing'' 
data.  The MLE of $\theta$ is the value $\hat{\theta}$ which maximizes 
the likelihood function 
\begin{equation}
\label{eqn:L}
L(\theta; y) = \int{ f(y, u; \theta) \mu( du) } \; 
\end{equation}
or, equivalently, the log likelihood $l(\theta; y) = \log L(\theta; y)$.  
The EM algorithm can be used to find $\hat{\theta}$ even if the integral 
in \eqref{eqn:L} is intractable.  Define the $Q$-function, a mapping on 
$\Theta \times \Theta$, by 
\begin{equation}
\label{eqn:Q}
Q( \theta | \tilde{\theta}; y) = \mathrm{E} \left\{ \log f(y, U ; \theta) 
~\big{\vert} ~y; \tilde{\theta} \right\} \; ,
\end{equation}
that is, the expected value of the ``complete data'' log-likelihood 
at $\theta$, given the observed data, this conditional expectation 
evaluated under $\tilde{\theta}$.  
Each EM iteration formally consists of an E-step, to evaluate the 
conditional expectation in \eqref{eqn:Q}, 
and an M-step, 
to maximize it.  More precisely, if $\theta^{(t)}$ is the parameter value 
as of the $t$th iteration, the update $\theta^{(t+1)}$ is chosen such 
that $Q(\theta^{(t+1)} | \theta^{(t)}; y) \geq Q(\theta | \theta^{(t)}; y)$ 
for all $\theta \in \Theta$.  Under regularity conditions 
\citep[][and see Section \ref{sec:EM} below]{boyl:1983,wu:1983}, and 
given a suitable starting value 
$\theta^{(0)}$, the resulting sequence $\left\{ \theta^{(t)}: 
t = 0, 1, \ldots \right\}$ will converge to a local maximizer of $L$.  

If the integral in \eqref{eqn:Q} admits a closed form solution, 
the implementation of EM is straightforward (though the M-step may 
still require a numerical optimization scheme such as Newton-Raphson).  
Suppose it does not.  As noted, the evaluation of \eqref{eqn:Q} 
requires taking an expectation with respect to the conditional 
distribution of the missing data $U$, given observed data $Y = y$.  If 
one has the means to simulate random draws from this target distribution, 
the $Q$-function can be approximated by Monte Carlo integration. 
Let $u^{(1)}, \ldots, u^{(m)}$ denote a random sample from $h(u | y; 
\tilde{\theta}) = f(y, u; \tilde{\theta})/L(\tilde{\theta}; y)$.  
Then a Monte Carlo approximation to \eqref{eqn:Q} is given by
$$
Q_m( \theta | \tilde{\theta}; y) = \frac{1}{m} \sum_{k=1}^m 
\log f(y, u^{(k)}; \theta) \; .
$$
In the Monte Carlo EM algorithm (MCEM), 
first introduced by \citet{wei:tann:1990}, the update $\theta^{(t+1)}$ is 
the value of $\theta$ that maximizes $Q_m(\theta | \theta^{(t)}; y)$.  

Applications of EM and MCEM have been numerous; in this work we focus 
on one in particular, the two-stage hierarchical model, introduced in 
Section \ref{sec:2shm}.  We give two simple but realistic examples from 
this class of models, and demonstrate the implementation of EM and MCEM 
in those two problems.  In Section \ref{sec:EM} we discuss convergence 
properties of the EM algorithm.  Of course, the question of convergence 
for MCEM is far more complicated, and an accessible discussion of the 
major results in this area is the main objective of this review paper.  
In Section \ref{sec:MCEM} we provide a rigorous but accessible review 
of the two seminal papers on MCEM convergence, those of 
\citet{chan:ledo:1995} and \citet{fort:moul:2003a}.  We make some 
concluding remarks in Section \ref{sec:conclusion}.

\section{Application: The two-stage hierarchical model}
\label{sec:2shm}

Let $Y = (Y_1, \ldots, Y_N)^T$, where each $Y_i$ is a random variable in 
$\mathbb{R}^1$, denote the observable data.  In a \emph{two-stage 
hierarchical model}, the distribution of $Y$ is specified conditionally on 
some unobservable random quantity $U = (U_1, \ldots, U_q)^T$.  Specifically, 
we assume that conditional on $U = u$, the $Y_i$ are independent with 
conditional densities denoted by $f_i(y_i | u_i; \theta_1)$, where 
$\theta_1 \in \Theta_1$ is an unknown parameter and each $f_i$ is a 
density with respect to Lebesgue or counting measure.  The $f_i$ may also 
depend on an observable covariate $x_i$ though this dependence is 
suppressed in our notation.  Define $f(y | u; \theta_1) = 
\prod_{i=1}^N f_i(y_i | u; \theta_1)$, a density on $\mathbb{R}^N$, and 
this completes specification of the first level, or \emph{stage}, of the 
hierarchy.  At the second stage we specify a marginal distribution for $U$, 
defined by $h(u; \theta_2)$, a density on $\mathbb{R}^q$ that depends on 
the unknown parameter $\theta_2 \in \Theta_2$.  Assume 
the parameter spaces $\Theta_1$ and $\Theta_2$ are open subsets of 
$\mathbb{R}^{d_1}$ and $\mathbb{R}^{d_2}$, respectively.  Let $d = d_1 + d_2$.  
The unknown parameter $\theta = (\theta_1, \theta_2)$ lies in the 
parameter space $\Theta = \Theta_1 \times \Theta_2$, an open subset of 
$\mathbb{R}^d$.  

Suppose we wish to compute maximum likelihood estimates (MLEs) of 
$\theta_1$ and $\theta_2$.  
Were the random effects $U$ observable the likelihood function would be 
given by what we will call the \emph{complete data likelihood}
$
L_c (\theta; y, u) = f(y | u; \theta_1) h(u; \theta_2) 
$.
But since only the data $Y$ are observed, the random effects must be 
integrated out of $L_C$ yielding the likelihood function
\begin{equation}
\label{eqn:likelihood}
L(\theta; y) = \int{ L_c( \theta; y, u) du } 
 = \int{ f(y | u; \theta_1) h(u; \theta_2) du } \; .
\end{equation}
We wish to find the value of $\theta$ that maximizes $L$, that is, the 
MLE $\hat{\theta}$.  

It will most often be the case that the integral in \eqref{eqn:likelihood} 
is intractable.  \citet{boot:hobe:jank:2001} provide a very nice summary 
of numerical and Monte Carlo methods available for maximum likelihood in 
this problem, arriving at the conclusion that ``Monte Carlo EM is generally 
the simplest and most efficient Monte Carlo fitting algorithm for 
two-stage hierarchical models.''  As noted above, the EM algorithm is a 
general method for maximum likelihood in the presence of missing data; 
hierarchical models are cast in this light by viewing the unobserved 
random effects as ``missing''.  

Let $l_c = \log L_c$ denote the complete data log likelihood, so
$$
l_c(\theta; y, u) = \log f(y | u; \theta_1) + \log h(u; \theta_2) \; .
$$
Thus in the setting of hierarchical models, the EM update rule introduced 
in Section \ref{sec:Intro} can be written
\begin{equation}
\label{eqn:EM.2shm}
\begin{split}
\theta_1^{(t+1)} & = \arg \max \mathrm{E} \left\{ \log f(y | U; \theta_1) 
\big{\vert} ~y; \theta^{(t)} \right\} \; , \\
\theta_2^{(t+1)} & = \arg \max \mathrm{E} \left\{ \log h( U; \theta_2) 
\big{\vert} ~y; \theta^{(t)} \right\} \; , \\
\end{split}
\end{equation}
that is, the update of $\theta_1$ and that of $\theta_2$ can be considered 
separately.

If one or both of the expectations in \eqref{eqn:EM.2shm} 
is intractable, one might employ the Monte Carlo 
EM algorithm.  The MCEM update rule for the two-stage hierarchical model 
is given here.  Let $\theta^{(t)} = (\theta_1^{(t)}, \theta_2^{(t)})$ denote 
the current parameter value; then $\theta^{(t+1)}$ is found by
\begin{enumerate}
\item Simulate $u^{(t,1)}, \ldots, u^{(t,m)}$, a random sample from the 
conditional density $h(u | y; \theta^{(t)})$; 
\item Compute updates
\[
\begin{split}
\theta_1^{(t+1)} & = \arg \max \left\{ \frac{1}{m} \sum_{k=1}^m \log 
f(y | u^{(t,k)}; \theta_1) \right\} \\
\theta_2^{(t+1)} & = \arg \max \left\{ \frac{1}{m} \sum_{k=1}^m \log 
h( u^{(t,k)}; \theta_2) \right\} \; . \\
\end{split}
\]
\end{enumerate}
The ``target density'' for the Monte Carlo E-step (step 1) is the 
conditional density of the random effects given the data,
\begin{equation}
\label{eqn:h.u.given.y}
h(u | y; \theta) \propto f(y | u; \theta_1) h(u; \theta_2) \; .
\end{equation}
If 
direct simulation from \eqref{eqn:h.u.given.y} is impossible, one might 
resort to a Markov chain Monte Carlo (MCMC) method such as the 
Metropolis-Hastings algorithm.  In this case the sample 
$\left\{ u^{(t,k)} : k = 1, \ldots, m \right\}$ is an ergodic Markov 
chain having $h(u | y; \theta^{(t)})$ as its unique stationary 
density  \citep[see, for example,][]{robe:case:2004}.  
An alternative approach is to compute a quasi-Monte Carlo or 
\emph{randomized quasi-Monte Carlo} \citep{lecu:lemi:2002} approximation 
to the $Q$-function with the goal of reducing Monte Carlo error and hence 
increasing the efficiency of the algorithm.  We will not consider quasi-Monte 
Carlo methods any further in this report; the interested reader is referred 
to \citet{jank:2004}.

\subsection{Example 1: A linear mixed model}

Table \ref{tab:bulls} contains a data set for an experiment described by 
\citet{sned:coch:1989}.  The experiment involved six bulls and very 
many cows.  From each bull, some number of semen samples was 
taken, and each of these samples was used in an attempt to artificially 
inseminate a large number of cows.  Some attempts were successful and some 
were not; let $Y_{ij}$ denote the success rate (percentage of conceptions) 
for sample $j$ from bull $i$, for $j = 1, \ldots, n_i$ and $i = 1, \ldots, 
q=6$; here $N = \sum_{i=1}^q n_i$.  Consider the one-way random effects model
$$
y_{ij} = \mu + u_i + e_{ij}
$$
where $\mu$ is the overall mean, $u_i$ is the $i$th bull effect, and 
$e_{ij}$ is a residual error term.  As the six bulls were a random sample 
from a larger population of bulls, the $u_i$ are modeled as independent 
and identically distributed (i$.$i$.$d$.$) random effects.  Model 
specification is completed by a distribution assumption on the bull effect 
and error term; we take
$$
u_i \sim ~\mathrm{iid}~ \mathrm{Normal} \left(0, \sigma_u^2 \right)\; ;  
~~~ \mathrm{independent}~ \mathrm{of}~~~ e_{ij} \sim ~\mathrm{iid}~ 
\mathrm{Normal} \left(0 , \sigma_e^2 \right) \; .
$$
When there exists a conjugate relationship between $f$ and $h$, as in the 
normal linear mixed model, the integral in \eqref{eqn:likelihood} can be 
solved explicitly.  The resulting log-likelihood can be maximized numerically 
(or analytically in the case of balanced data $n_i \equiv n$); for the bulls 
data we obtain $\hat{\mu} = 53.318$, $\hat{\sigma}_u^2 = 54.821$, and 
$\hat{\sigma}_e^2 = 249.23$.  

\bigskip
\bigskip
\begin{table}[h]
\begin{center}
\begin{tabular}{ l l l }
\hline
Bull ($i$) ~~~~~~~& $n_i$ ~~~~~~~~~~& Percentage of conception \\
\hline
~~~1 & ~5 & ~46, 31, 37, 62, 30 \\
~~~2 & ~2 & ~70, 59 \\
~~~3 & ~7 & ~52, 44, 57, 40, 67, 64, 70 \\
~~~4 & ~5 & ~47, 21, 70, 46, 14 \\
~~~5 & ~7 & ~42, 64, 50, 69, 77, 81, 87 \\
~~~6 & ~9 & ~35, 68, 59, 38, 57, 76, 57, 29, 60 ~~~~~~~\\
 Total & 35 & \\
\hline
\end{tabular}
\caption{\textit{Bovine artificial insemination data of Example 1
         \citep{sned:coch:1989}.}}  
\label{tab:bulls}
\end{center}
\end{table} 

Consider the EM algorithm.  We find it more convenient to work with an 
equivalent version of the model in which $y_{ij} = u_i + e_{ij}$ and the 
$u_i$ are i$.$i$.$d$.$ $\mathrm{Normal}(\mu, \sigma_u^2)$.  Under this 
reparameterization the complete data log-likelihood of 
$\theta = (\mu, \sigma_u^2 , \sigma_e^2)$ is
$$
l_c(\theta; y, u) = - \frac{N}{2} \log( \sigma_e^2 ) - \frac{1}{2 \sigma_e^2} 
\sum_{i=1}^q \sum_{j=1}^{n_i}(y_{ij} - u_i )^2 - \frac{q}{2} \log( \sigma_u^2 ) 
- \frac{1}{2 \sigma_u^2} \sum_{i=1}^q (u_i - \mu)^2 \; .
$$
Owing to the conjugacy it is straightforward to show that
\begin{equation}
\label{eqn:posterior.bulls}
U_i | (Y = y; \theta) ~~i = 1, \ldots, q ~ ~\mathrm{are} ~ \mathrm{indep} 
~~ \mathrm{Normal} \left( \frac{ \sigma_e^2 \mu + n_i \sigma_u^2 \bar{y}_i }
{ \sigma_e^2 + n_i \sigma_u^2 } , \frac{ \sigma_e^2 \sigma_u^2 }
{ \sigma_e^2 + n_i \sigma_u^2 } \right) \; .
\end{equation}
Denote the conditional mean and variance of $U_i$ given $Y=y$ by 
$\hat{u}_i$ and $\hat{V}_i$, respectively.  Then the EM update rule is 
given by
\[
\begin{split}
\mu^{(t+1)}  & =  \frac{1}{q} \sum_{i=1}^q \hat{u}_i^{(t)} \\
\sigma_u^{2^{(t+1)}} & =  \frac{1}{q} \sum_{i=1}^q \left( \hat{V}_i^{(t)} + 
\left[ \hat{u}_i^{(t)} \right]^2 \right) - \left[ \mu^{(t+1)} \right]^2 \\
\sigma_e^{2^{(t+1)}} & =  \frac{1}{N} \sum_{i=1}^q \left[ \sum_{j=1}^{n_i} 
y_{ij}^2 - 2 n_i \bar{y}_i \hat{u}_i^{(t)} + n_i \left( \hat{V}_i^{(t)} 
+ \left[ \hat{u}_i^{(t)} \right]^2 \right) \right] \; . \\
\end{split}
\]
Given the existence of a closed form EM update, there is no practical reason 
to resort to Monte Carlo EM (indeed there was no practical need for EM, as 
we found a closed form expression for the likelihood as well), but we will 
consider MCEM for illustration.  Let $u^{(t,1)}, \ldots, u^{(t,m)}$ denote a 
sequence of simulated draws from $h(u | y; \theta^{(t)})$, given at 
\eqref{eqn:posterior.bulls}.  The MCEM update rule for $\theta^{(t+1)}$ is
\[
\begin{split}
\mu^{(t+1)}  & =  \frac{1}{mq} \sum_{k=1}^m \sum_{i=1}^q u_i^{(t,k)} \\
\sigma_u^{2^{(t+1)}} & =  \frac{1}{mq} \sum_{k=1}^m 
\sum_{i=1}^q \left( u_i^{(t,k)} - \mu^{(t+1)} \right)^2 \\
\sigma_e^{2^{(t+1)}} & =  \frac{1}{mN} \sum_{k=1}^m \sum_{i=1}^q \sum_{j=1}^{n_i} 
 \left( y_{ij} - u_i^{(t,k)} \right)^2  \; . \\
\end{split}
\]
We ran three independent MCEM runs of 20 iterations each, starting at the 
point $(\mu^{(0)}, \sigma_u^{2^{(0)}},\sigma_e^{2^{(0)}} ) = (55, 45, 260)$.  For 
each update we used Monte Carlo sample size $m = 10^4$; results are shown 
in Figure \ref{fig:bulls1}.  The three dashed lines indicate the paths of 
the three MCEM runs, and the solid line shows that of ordinary (deterministic) 
EM.  We did three more runs with starting values closer to the MLE and using 
$m = 10^5$; those results are summarized in Figure \ref{fig:bulls2}. 

\begin{figure}[h]
\begin{center}
 \includegraphics[scale=.5,angle=360]{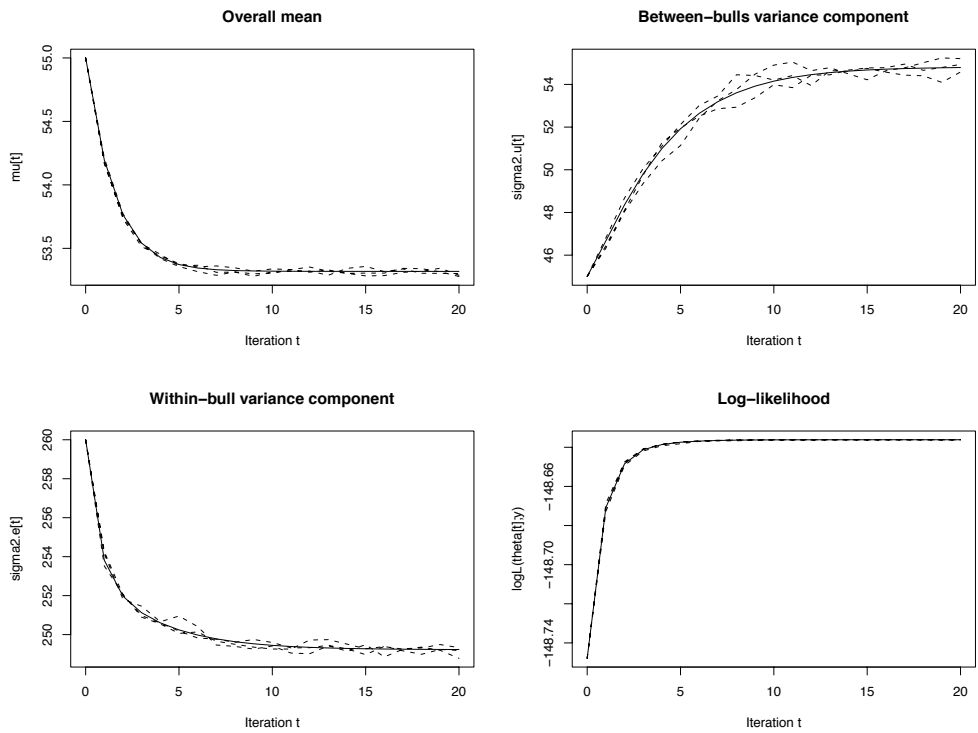}
\end{center}
\caption{\textit{Trace plots for Monte Carlo EM in Example 1, based on 
Monte Carlo sample size $m = 10^4$ at each iteration.  Top left plot is 
overall mean $\mu$, top right and bottom left are variance components 
$\sigma_u^2$ and $\sigma_e^2$, respectively.  Bottom right plot shows 
log-likelihood evaluated at current parameter value.  The solid line 
is deterministic EM and the three dashed lines correspond to three 
independent runs of Monte Carlo EM. }}
\label{fig:bulls1}
\end{figure}
\bigskip

\begin{figure}[h]
\begin{center}
  \includegraphics[scale=.5,angle=360]{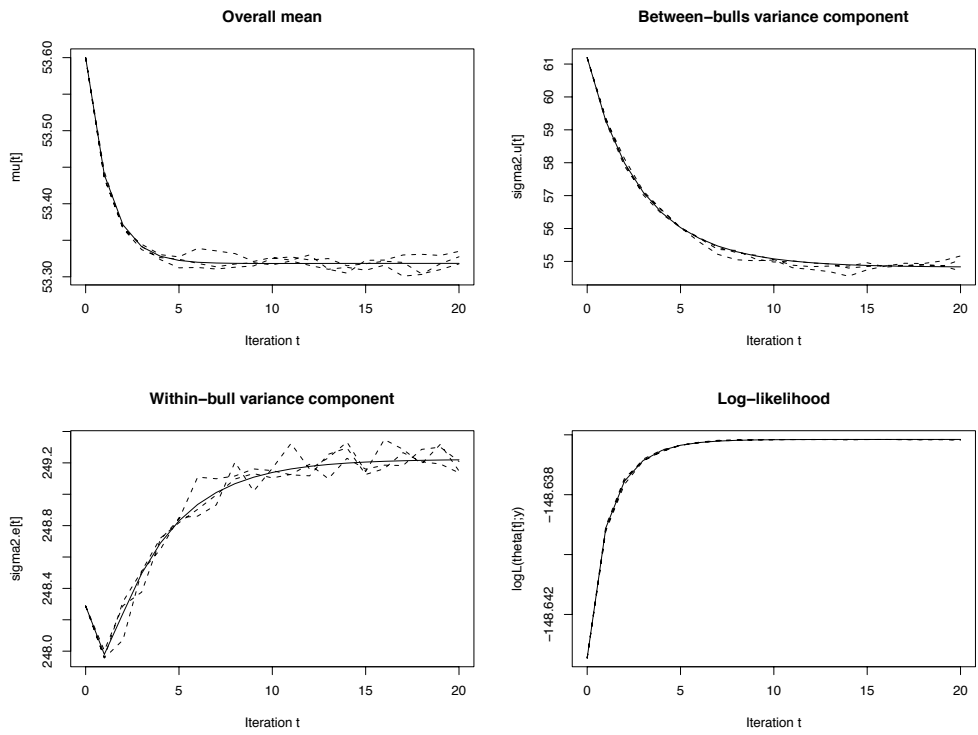}
\end{center}
\caption{\textit{Analogous to Figure \ref{fig:bulls1}, but with 
$m = 10^5$ and starting 
values chosen closer to the true MLE. 
}}
\label{fig:bulls2}
\end{figure}

\FloatBarrier

\subsection{Example 2: A logit-normal generalized linear mixed model}

Let $Y = \left\{ Y_{ij} : j = 1, \ldots, n_i; ~ i = 1, \ldots, q \right\}$ 
denote a set of binary response variables; here again one can think of 
$Y_{ij}$ as the $j$th response for the $i$th subject.  Let $x_{ij}$ be a 
covariate (or vector of covariates) associated with the $i,j$ observation.  
Conditional on the random effects $U = u \in \mathbb{R}^q$, the responses 
are independent $\mathrm{Bernoulli}(\pi_{ij})$ where
$$
\log \left( \frac{ \pi_{ij} }{ 1 - \pi_{ij} } \right) = \beta x_{ij} + u_i 
\; .
$$
Let $U_1, \ldots, U_q$ be independent and identically distributed as 
$\mathrm{Normal}(0, \sigma^2)$.  The likelihood is given by
\[
\begin{split}
L(\beta, \sigma^2; y) & =  \left( \sigma^2 \right)^{-q/2}  \times \\
&  ~~
\int_{ \mathbb{R}^q } \exp \left\{ \sum_{i=1}^q \sum_{j=1}^{n_i} \left[  y_{ij} 
\left( \beta x_{ij} + u_i \right) - \log \left( 1 + e^{ \beta x_{ij} + u_i }
\right) \right] - \frac{1}{2 \sigma^2} \sum_{i=1}^q u_i^2 \right\} du \; . \\
\end{split}
\]
The above model has been used by several authors  
\citep{mccu:1997, boot:hobe:1999, caff:jank:jone:2005} as a benchmark for 
comparing Monte Carlo methods of maximum likelihood.  We consider here a 
data set generated by \citet[][Table 2]{boot:hobe:1999} with $n_i = 15$, 
$q = 10$, and $x_{ij} = j/15$ for each $i,j$.  For these data the MLEs 
are known to be $( \hat{\beta}, \hat{\sigma}^2 ) = (6.132, 1.766)$.  

A version of the 
complete data log-likelihood is given by
$$
l_c(\beta, \sigma^2; y, u) = - \frac{q}{2} \log \left(\sigma^2 \right) 
- \frac{1}{2 \sigma^2} 
\sum_{i=1}^q u_i^2 + \sum_{i=1}^q \sum_{j=1}^{n_i} \left[ \beta x_{ij} y_{ij} 
- \log \left( 1 + e^{\beta x_{ij} + u_i } \right) \right] \; .
$$
To apply the EM algorithm in this problem we would need to compute the 
(conditional) expectation of $l_c$ with respect to the density
\begin{equation}
\label{eqn:benchmark.target}
h(u | y; \theta) \propto \exp \left\{ \sum_{i=1}^q \sum_{j=1}^{n_i} 
\left[y_{ij} u_i - \log \left( 1 + e^{\beta x_{ij} + u_i } \right) 
\right] - \frac{1}{2 \sigma^2} \sum_{i=1}^q u_i^2 \right\} \; .
\end{equation}
Clearly this integral will be intractable.  Thus we consider a Monte 
Carlo EM algorithm, which requires the means to simulate random draws 
from the distribution given by \eqref{eqn:benchmark.target}.  
\citet{mccu:1997} employed a variable-at-a-time Metropolis-Hastings 
independence sampler with $\mathrm{Normal}(0, \sigma^2)$ proposals, 
which \citet{john:jone:neat:2011} have shown is uniformly ergodic.  

Trace plots for three independent runs of MCEM are shown in the 
left hand panels of Figure \ref{fig:benchmark}.  The starting values for 
these runs were $(\beta^{(0)}, \sigma^{2(0)}) = (2,1)$, and we ran 35 
updates with Monte Carlo sample size $m = 10^4$ at each iteration.  We 
conducted three more runs of 25 iterations, starting at 
$(\beta^{(0)}, \sigma^{2(0)}) = (6,2)$, with $m = 10^5$; results are 
shown in the right hand panels of Figure \ref{fig:benchmark}.


\begin{figure}[h]
\begin{center}
  \includegraphics[scale=.5,angle=360]{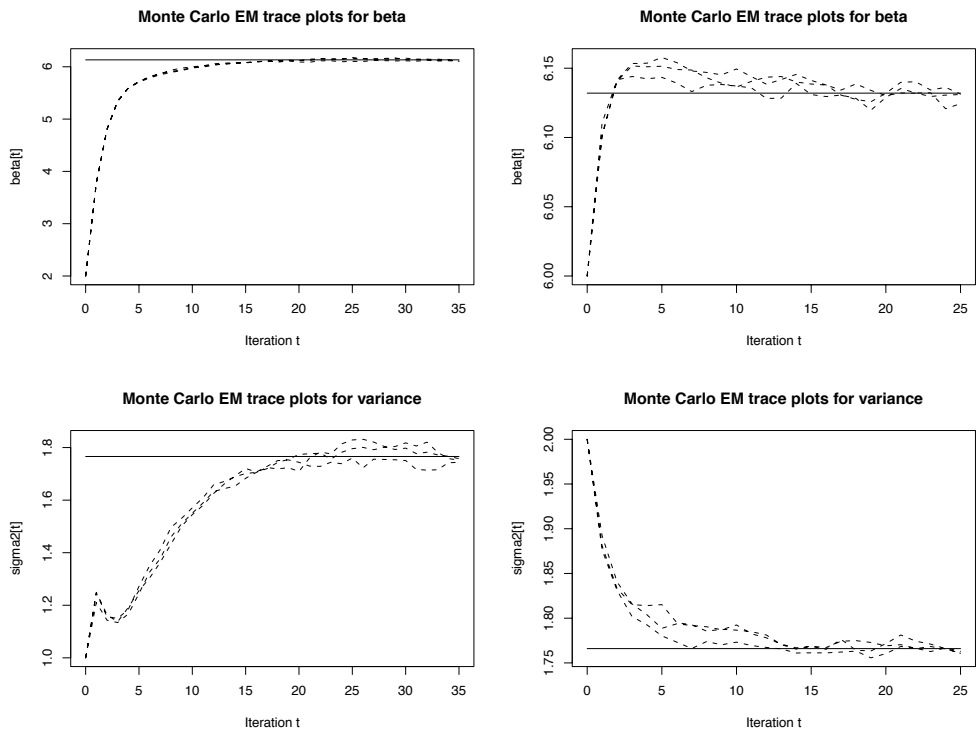}
\end{center}
\caption{\textit{Monte Carlo EM trace plots for logit-normal model of 
Example 2.  Top panels show $\beta$, bottom panels show $\sigma^2$.  
Three dashed lines correspond to three independent runs of MCEM, with solid 
horizontal line drawn at true MLE.  Runs in left hand panels used Monte 
Carlo sample size $m = 10^4$ at each iteration; in right hand panels we 
used $m=10^5$ with starting values closer to the true MLE. }}
\label{fig:benchmark}
\end{figure}



\FloatBarrier

\section{Convergence properties of ordinary EM}
\label{sec:EM}

The basic convergence properties of the EM algorithm were established 
by \citet{boyl:1983} and \citet{wu:1983}.  The presentation given here 
draws heavily from \citet{geye:1998}.   We will show that if an EM 
sequence converges, its limit must be a stationary point of the 
log-likelihood.  We then present conditions that guarantee the 
convergence of EM, with additional conditions that guarantee convergence 
to the MLE.  We conclude this section with a proof that the EM algorithm 
cannot produce a superlinearly convergent sequence.

We begin by proving the \emph{ascent property} of the EM algorithm, which 
guarantees that an EM update will never decrease the value of the likelihood 
function, that is, if $\left\{ \theta^{(t)} \right\}$ is an EM sequence, 
then $l(\theta^{(t+1)}; y) \geq l(\theta^{(t)}; y)$ for each $t$. 

Define 
\begin{equation}
\label{eqn:R}
\begin{split}
R(\theta | \tilde{\theta}; y) & = \mathrm{E} \left\{ \log h(U | y; \theta) 
\big{\vert}~ y; \tilde{\theta} \right\} \\ &  = \mathrm{E} \left\{ \log f(y, U; 
\theta) \big{\vert}~ y; \tilde{\theta} \right\} - 
\mathrm{E} \left\{ \log f(y; \theta) \big{\vert}~ y; \tilde{\theta} \right\}
\\ & = Q(\theta | \tilde{\theta}; y) - l(\theta; y) \; .
\end{split}
\end{equation}
We now show that, for fixed $\tilde{\theta}$, $R(\theta | 
\tilde{\theta}; y)$ attains its maximum at $\theta = \tilde{\theta}$.
\begin{lemma}
\label{lem:R}
For any $\tilde{\theta} \in \Theta$, $R(\tilde{\theta} | \tilde{\theta}; y) 
\geq R(\theta | \tilde{\theta}; y)$ for all $\theta$.  
\end{lemma}
\begin{proof}
\[
R(\theta | \tilde{\theta} ; y) - R(\tilde{\theta} | \tilde{\theta}; y) 
= \mathrm{E} \left\{ \log \left( \frac{h(U | y; \theta)}{h(U | y; 
\tilde{\theta})} \right) \Big{\vert} ~ y; ~\tilde{\theta} \right\} 
\leq \log \left( \mathrm{E} \left\{ \frac{h(U | y; \theta)}{h(U | y; 
\tilde{\theta})} \Big{\vert} ~y; ~\tilde{\theta} \right\} \right)
\]
by the conditional Jensen inequality \citep[see][page 449]{bill:1995}; now
$$
\mathrm{E} \left\{ \frac{h(U | y; \theta)}{h(U | y; 
\tilde{\theta})} \Big{\vert} ~y; ~\tilde{\theta} \right\} = \int 
\frac{h(u | y; \theta)}{h(u | y; \tilde{\theta})} h(u | y; \tilde{\theta}) 
du = \int h(u | y; \theta) du = 1 
$$  
and thus $R(\theta | \tilde{\theta} ; y) - R(\tilde{\theta} | \tilde{\theta} 
; y) \leq \log(1) = 0$.  
\end{proof}
\begin{theorem} 
\label{thm:ascent}
If $Q(\theta | \tilde{\theta}; y) \geq 
Q(\tilde{\theta} | \tilde{\theta}; y)$, then $l(\theta; y) \geq 
l(\tilde{\theta}; y)$.  
If $Q(\theta | \tilde{\theta}; y) > 
Q(\tilde{\theta} | \tilde{\theta}; y)$, then $l(\theta; y) > 
l(\tilde{\theta}; y)$.  
\end{theorem}
\begin{proof}
By \eqref{eqn:R} and Lemma \ref{lem:R},
\[
\begin{split}
l(\theta; y) - l(\tilde{\theta} ; y) & = Q(\theta | \tilde{\theta}; y) 
- Q(\tilde{\theta} | \tilde{\theta} ; y) - \left[ R(\theta | 
\tilde{\theta}; y) - R(\tilde{\theta} | \tilde{\theta} ; y) \right] \\ & 
\geq Q(\theta | \tilde{\theta}; y) - Q(\tilde{\theta} | \tilde{\theta} ; y)
\end{split}
\]
\end{proof}
\noindent The ascent property of EM follows immediately from Theorem 
\ref{thm:ascent}: since $\theta^{(t+1)}$ is chosen to maximize $Q(\theta |
\theta^{(t)}; y)$, it must be that $Q(\theta^{(t+1)} | \theta^{(t)}; y) \geq 
Q(\theta^{(t)} | \theta^{(t)}; y)$ and thus $l(\theta^{(t+1)}; y) \geq 
l(\theta^{(t)}; y)$.  This is an appealing property, as it guarantees that 
an EM update will never take a step in the wrong direction.  Of course, this 
result tells us absolutely nothing about the convergence of an EM sequence.

We now show that if an EM sequence converges, it converges to a stationary 
point of the log-likelihood.  Unless otherwise noted, $\nabla$ will denote 
differentiation with respect to the first argument.  
\begin{theorem}
\label{thm:em.thm2}
Suppose the mapping $(\theta, \tilde{\theta}) \mapsto \nabla Q(\theta 
| \tilde{\theta}; y)$ is jointly continuous.  If $\theta^*$ is the limit of 
an EM sequence $\left\{ \theta^{(t)} \right\}$, then $\nabla l(\theta^*; y) 
= 0$.
\end{theorem}  
\begin{proof}
Since $\theta^{(t+1)}$ maximizes $Q(\theta | \theta^{(t)}; y)$ at each $t$ we 
have $\nabla Q(\theta^{(t+1)} | \theta^{(t)}; y) = 0$ at each $t$.  By the 
continuity assumption $\nabla Q(\theta^{(t+1)} | \theta^{(t)}; y) \rightarrow 
\nabla Q(\theta^* | \theta^*; y)$ as $t \rightarrow \infty$ and thus 
$\nabla Q(\theta^* | \theta^*; y) = 0$.  
Let $R$ be as defined at \eqref{eqn:R}, and note
\[
\begin{split}
\nabla R(\theta | \theta; y) & = \int \left[ \frac{\partial}{\partial \theta}
\log h(u | y; \theta) \right] h(u | y; \theta) du  \\
 & = \int \frac{ \frac{\partial}{\partial \theta} h(u | y; \theta)} 
 { h(u | y; \theta) } h(u | y; \theta) du \\ 
 & = \frac{\partial}{\partial \theta} \int h(u | y; \theta) du 
 = \frac{\partial}{\partial \theta} ( 1 ) = 0 \; .
\end{split}
\]
It then follows from \eqref{eqn:R} that
$$
\nabla l(\theta^*; y) = \nabla Q(\theta^* | \theta^*; y) = 0 \; .
$$  
\end{proof}
From Theorem \ref{thm:em.thm2} we have that \emph{if} the EM algorithm 
converges, it converges to a stationary point of $l$; we as yet have no 
guarantee that EM converges.  By the ascent property, the limit of an EM 
sequence (if it exists) cannot be a local minimum.  It can, however, be 
a local but not global maximum \citep[][cites several examples]{wu:1983} 
or a saddlepoint \citep[][gives an example]{murr:1977}.  

We will now specify conditions that do guarantee the convergence of the 
EM algorithm.  We define a generalized EM sequence as one in which each 
update increases the $Q$-function, but does not necessarily maximize it.
\begin{definition}
A \emph{generalized EM (GEM) sequence} is a sequence of parameter values 
$\left\{ \theta^{(t)} \right\}$ satisfying $Q(\theta^{(t+1)} | \theta^{(t)}; y) 
\geq Q(\theta^{(t)} | \theta^{(t)}; y)$ for each $t$.
\end{definition}
It is immediately clear from Theorem \ref{thm:ascent} that a GEM sequence 
enjoys the ascent property $l(\theta^{(t+1)}; y) \geq l(\theta^{(t)}; y)$.  
The conclusion of Theorem \ref{thm:em.thm2}, that the limit of an EM 
sequence (if it exists) must be a stationary point of $l$, does not hold 
for GEM without additional assumptions.

Consider a sequence of parameter values $\left\{ \theta^{(t)} \right\}$ 
satisfying $\theta^{(t+1)} \in M(\theta^{(t)})$ for some point-to-set 
mapping $M$.  For example, a GEM sequence can be formulated in this 
manner by taking $M(\tilde{\theta}) = \left\{ \theta : 
Q(\theta | \tilde{\theta}; y) \geq Q(\tilde{\theta} | 
\tilde{\theta}; y) \right\}$.  We will indicate a point-to-set mapping 
$M$ in $\Theta$ by the notation $M : \Theta \rightrightarrows \Theta$.  
\begin{definition}
The point-to-set mapping $M: \Theta \rightrightarrows \Theta$ is 
\emph{outer semicontinuous} if the graph of $M$, $$\left\{ (\theta, 
\tilde{\theta}) \in \Theta \times \Theta : \theta \in M(\tilde{\theta}) 
\right\}$$ is a closed set; that is, if for any convergent sequence 
$\left\{ (\theta^{(t)}, \tilde{\theta}^{(t)}) \right\}$ satisfying 
$\theta^{(t)} \in M(\tilde{\theta}^{(t)} )$ for each $t$, the limit 
$(\theta^*, \tilde{\theta}^*)$ satisfies $\theta^* \in M(\tilde{\theta}^*)$.  
\end{definition}
The following theorem gives a set of conditions under which every cluster 
point of a GEM sequence lies in a particular set $\Gamma \subset \Theta$.
\begin{theorem}
\label{thm:em.thm3}
Let $\Gamma \subset \Theta$ and $M : \Theta \rightrightarrows \Theta$ be such 
that the following conditions hold.
\begin{enumerate}
\item $M(\tilde{\theta}) \subset \left\{ \theta : Q(\theta | \tilde{\theta}; 
y ) \geq Q( \tilde{\theta} | \tilde{\theta}; y) \right\}$ when $\tilde{\theta} 
\in \Gamma$.
\item $M(\tilde{\theta}) \subset \left\{ \theta : Q(\theta | \tilde{\theta}; 
y ) > Q( \tilde{\theta} | \tilde{\theta}; y) \right\}$ when $\tilde{\theta} 
\in \Theta ~\backslash~ \Gamma$.
\item The restriction of $M$ to $\Theta ~\backslash~ \Gamma$ is outer 
semicontinuous.
\end{enumerate}
Further suppose that the log-likelihood $l$ is continuous, that the level 
set \\ $\left\{ \theta : l(\theta; y) \geq l(\theta^{(0)}; y) \right\}$ is 
compact, and let the sequence $\left\{ \theta^{(t)} : t = 0, 1, 2, \ldots 
\right\}$ be such that $\theta^{(t+1)} \in M(\theta^{(t)})$ for each $t$.  
Then $l(\theta^{(t)}; y)$ converges to a limit, and every cluster point of 
$\left\{ \theta^{(t)} \right\}$ is contained in $\Gamma$.  
\end{theorem}
\begin{proof}
By assumption the log-likelihood is bounded above.  Also, $\left\{ 
\theta^{(t)} \right\}$ is a GEM sequence, hence $l(\theta^{(t)}; y)$ is 
nondecreasing, so it converges to a limit $\lambda$.  

Suppose to get a contradiction there exists a subsequence 
$\theta^{(t_k)} \rightarrow \theta^* \notin \Gamma$.  Consider the subsequence 
$\left\{ \theta^{(t_k+1)} \right\}$.  By the ascent property $l(\theta^{(t_k + 1)}; 
y) \geq l(\theta^{(0)}; y)$ for each $k$, so the compactness assumption 
guarantees that $\left\{ \theta^{(t_k+1)} \right\}$ has a convergent 
subsequence with limit $\theta^{**}$.  Further, $\theta^{**} \in 
M(\theta^*)$ by the outer semicontinuity of $M$, and thus $Q(\theta^{**} | 
\theta^*; y) > Q(\theta^* | \theta^*; y)$ and thus 
$l(\theta^{**}; y) > l(\theta^*; y)$ by assumption 2 and Theorem 
\ref{thm:ascent}, respectively.  But $l(\theta^{**}; y) = \lambda = 
l(\theta^*; y)$ by continuity of $l$, a contradiction.  

Thus all cluster points of $\left\{ \theta^{(t)} \right\}$ are in $\Gamma$.  
\end{proof}
In the obvious application of Theorem \ref{thm:em.thm3} the solution set 
$\Gamma$ is taken to be the set of stationary points of the log-likelihood.  
We now have a set of conditions under which the EM algorithm is guaranteed 
to converge to the unique MLE $\hat{\theta}$.
\begin{corollary}
\label{cor:unique.mle}
If the conditions of Theorem \ref{thm:em.thm3} hold and the set $\Gamma$ 
consists of a single point $\hat{\theta}$, then the sequence $\left\{ 
\theta^{(t)} \right\}$ converges to $\hat{\theta}$.  
\end{corollary}
Unfortunately, these conditions can be difficult or impossible to verify 
in many practical applications.  Further, the rate of convergence of the 
EM algorithm cannot be superlinear, as we show here.  
\begin{definition}
The sequence $\left\{ \theta^{(t)} \right\}$ converging to $\hat{\theta}$ 
is said to converge \emph{superlinearly} if 
$$
\theta^{ (t+1) } - \hat{\theta} = o \left( || \theta^{(t)} - 
\hat{\theta} || \right)
$$
as $t \rightarrow \infty$, where $|| \cdot ||$ denotes the standard 
Euclidean norm.
\end{definition}
\begin{lemma}
\label{lemma:em.gradient}
Suppose the log-likelihood is twice continuously differentiable with a 
local maximum at $\hat{\theta}$ and suppose that $\nabla^2 l(\hat{\theta}; 
y)$ is nonsingular and negative definite.  Further supose that 
$\nabla^2 Q( \hat{\theta} | \hat{\theta}; y)$ is nonsingular and negative 
definite and $\nabla^2 Q(\hat{\theta} | \hat{\theta};y) - \nabla^2 
l(\hat{\theta}; y)$ is nonsingular.  
Define the sequence $\left\{ \theta^{(t)} \right\}$ by
\begin{equation}
\label{eqn:emg}
\theta^{(t+1)} = \theta^{(t)} - \left[ \nabla^2 Q(\theta^{(t)} | 
\theta^{(t)}; y) \right]^{-1} \nabla Q(\theta^{(t)} | \theta^{(t)}; y)
\end{equation}
and suppose that $\theta^{(t)} \rightarrow \hat{\theta}$.  Then the 
convergence is not superlinear.
\end{lemma}
\begin{proof}
Let $\delta_{NR}$ denote the Newton-Raphson update increment for the 
optimization of $l$, that is, if $\left\{ \theta'^{(t)} \right\}$ is a 
Newton-Raphson sequence then $\theta'^{(t+1)} = \theta'^{(t)} + 
\delta_{NR}(\theta'^{(t)})$ for each $t$, or
\begin{equation*}
\delta_{NR}(\theta) = - \left[ \nabla^2 l(\theta; y) \right]^{-1} 
\nabla l(\theta; y) \; .
\end{equation*}
Since $\nabla^2 l(\theta; y)$ is continuous and $\nabla^2 l(\hat{\theta};y)$ 
is nonsingular, it must be that $\nabla^2 l(\theta; y)$ is invertible in a 
neighborhood of $\hat{\theta}$, and thus $\delta_{NR} ( \theta^{(t)} )$ is 
well-defined for sufficiently large $t$.  

By convergence of $\left\{ \theta^{(t)} \right\}$ and the continuity of 
$\nabla l$, $\nabla l(\theta^{(t)}; y) \rightarrow \nabla l(\hat{\theta}; 
y) = 0$.   Together with the continuity of $\nabla^2 l(\theta; y)$, this 
guarantees that
$$
\delta_{NR}( \theta^{(t)} ) = - \left[ \nabla^2 l( \theta^{(t)}; y) 
\right]^{-1} \nabla l(\theta^{(t)}; y) \rightarrow \left[ \nabla^2 
l(\hat{\theta}; y) \right]^{-1} \cdot 0 = 0
$$
as $t \rightarrow \infty$.  Now, consider the sequence $\left\{ \nabla 
l(\theta^{(t)}; y) / || \nabla l(\theta^{(t)}; y) || \right\}$.  This 
sequence lives on the unit sphere, a compact set, and hence has a convergent 
subsequence.  Let $\left\{t_k \right\}$ denote the indices of a convergent 
subsequence and $b$ its limit.  Then
\begin{equation}
\label{eqn:subseq1}
\frac{ \theta^{(t_k+1)} - \theta^{(t_k)} }{ || \nabla l(\theta^{(t_k)}; y) 
|| } = \frac{ - \left[ \nabla^2 Q(\theta^{(t_k)} | \theta^{(t_k)}; y) 
\right]^{-1} \nabla l(\theta^{(t_k)}; y) }{ || \nabla l(\theta^{(t_k)}; y) 
|| } \rightarrow - \left[ \nabla^2 Q(\hat{\theta} | \hat{\theta}; y) 
\right]^{-1} b
\end{equation}
and
\begin{equation}
\label{eqn:subseq2}
\frac{ \delta_{NR} ( \theta^{(t_k)}) }{ || \nabla l(\theta^{(t_k)}; y) 
|| } = \frac{ - \left[ \nabla^2 l(\theta^{(t_k)} ; y) 
\right]^{-1} \nabla l(\theta^{(t_k)}; y) }{ || \nabla l(\theta^{(t_k)}; y) 
|| } \rightarrow - \left[ \nabla^2 l(\hat{\theta}  y) 
\right]^{-1} b
\end{equation}
as $k \rightarrow \infty$.  The equality in \eqref{eqn:subseq1} follows 
from the fact that $\nabla Q(\theta | \theta; y) = \nabla l(\theta; y)$ 
for any $\theta$.  

Suppose the sequence $\left\{ \theta^{(t)} \right\}$ does converge 
superlinearly.  Then it is asymptotically equivalent to Newton-Raphson 
by the Dennis-Mor\'{e} characterization theorem 
\citep[see, for example,][]{flet:1987}, and thus the (sub)sequences defined 
in \eqref{eqn:subseq1} and \eqref{eqn:subseq2} must have the same limit.  
Then
$
\left[ \nabla^2 Q(\hat{\theta} | \hat{\theta}; y) \right]^{-1} b = 
\left[ \nabla^2 l(\hat{\theta} ; y) \right]^{-1} b = c$.  
So
$$
\left[ \nabla^2 Q(\hat{\theta} | \hat{\theta}; y) - 
 \nabla^2 l(\hat{\theta} ; y) \right] c = 0
$$
and thus $c = 0$ since $\nabla^2 Q(\hat{\theta} | \hat{\theta}; y) - 
 \nabla^2 l(\hat{\theta} ; y)$ is full rank.  But $b$ must be on the 
unit sphere, a contradiction.

Thus the convergence of $\left\{ \theta^{(t)} \right\}$ to 
$\hat{\theta}$ is not superlinear.  
\end{proof}
The algorithm defined at \eqref{eqn:emg}, with 
update rule given by 
a single Newton-Raphson iteration toward the 
maximum of the $Q$-function, was first introduced by \citet{lang:1995} 
and is known as the {\em EM gradient algorithm}.  Details are beyond the 
scope of this report, but roughly speaking, the convergence properties of 
the EM algorithm are equally enjoyed by \pcite{lang:1995} EM gradient 
algorithm.  Thus while Lemma \ref{lemma:em.gradient} takes the convergence 
of the EM gradient sequence as a given, there is no sacrifice in the 
applicability of the result, as the EM gradient converges to a local 
maximum under essentially the same conditions as does the EM algorithm.
\begin{theorem}
Suppose the EM sequence $\left\{ \theta^{(t)} \right\}$ converges to a point 
$\theta^* \in \Theta$, a stationary point of the log-likelihood.  Further 
suppose that $l(\theta; y)$, $Q(\theta | \tilde{\theta}; y)$, and 
$R(\theta | \tilde{\theta}; y)$ are twice continuously differentiable in 
$\theta$ and that $\nabla^2 l(\theta^*; y)$, $\nabla^2 Q(\theta^* | 
\theta^*; y)$, and $\nabla^2 R(\theta^* | \theta^*; y)$ have full rank.  
Then the convergence cannot be superlinear.  
\end{theorem}
\begin{proof}
Let $\delta_{EG}$ denote the EM gradient update increment, that is, if 
$\left\{ \theta'^{(t)} \right\}$ is an EM gradient sequence then 
$\theta'^{(t+1)} = \theta'^{(t)} + \delta_{EG}(\theta'^{(t)})$ for each $t$:  
$$
\delta_{EG}(\theta) = - \left[ \nabla^2 Q(\theta | \theta ;y) \right]^{-1}
\nabla Q(\theta | \theta; y) \; .
$$
Define $\delta_{EM}$ analogously, so $\theta + \delta_{EG}(\theta)$ represents 
the first iteration in a Newton-Raphson routine starting at $\theta$ and 
converging to $\theta + \delta_{EM}(\theta)$.  Since Newton-Raphson converges 
superlinearly in this subproblem 
\citep[see, for example,][Theorem 3.1.1]{flet:1987}, we have
$$
\theta + \delta_{EG}(\theta) - \left[ \theta + \delta_{EM}(\theta) 
\right] = o \left( || \delta_{EM}(\theta) || \right)
$$
or
$$
\delta_{EG}(\theta) = \delta_{EM}(\theta) + o \left( || \delta_{EM} 
(\theta) || \right) \; ,
$$
and thus the EM gradient algorithm \eqref{eqn:emg} is asymptotically 
equivalent to the EM algorithm.  But EM gradient is not superlinearly 
convergent by Lemma \ref{lemma:em.gradient}, and thus neither 
is the EM algorithm.  
\end{proof}

\section{Some convergence results for Monte Carlo EM}
\label{sec:MCEM}

It seems a statement of the obvious (and an understatement at that) to 
point out that the study of convergence properties of Monte Carlo EM 
is more complicated than that of ordinary EM.  Even before coming to face 
the complexity of the mathematical arguments, one must determine 
which notion of ``convergence'' one wishes to consider -- what exactly 
is going to infinity?  We mention here three distinct approaches to 
the problem.  

The first serious effort in establishing convergence properties of MCEM 
is that of \citet{chan:ledo:1995}, who treat the data as fixed, and hold 
the Monte Carlo sample size $m$ constant across MCEM iterations.  They 
then let $m$ go to infinity, and study the asymptotic properties of the 
MCEM \emph{sequence} as a Monte Carlo approximation to the ordinary 
EM sequence with the same starting value (whose convergence properties 
are well understood).  We will discuss \pcite{chan:ledo:1995} results 
in considerable detail in subsection \ref{sub:chan.ledo}.  
On the other hand, unless the Monte Carlo sample size is allowed to 
increase with the iteration count, there is no chance for convergence 
in the usual sense (convergence to the MLE) because of persistent 
Monte Carlo error.

In the version of MCEM considered by \citet{sher:ho:dala:1999}, the 
Monte Carlo E-step is carried out by running multiple (independent) 
Markov chains generated by a Gibbs sampler.  Their theoretical results 
are built on allowing the number of chains, the length of each chain, 
and the number of EM iterations $T$ to all tend to infinity, as does the 
data sample size $N$.  They then prove $\sqrt{N}$-consistency and 
asymptotic normality of the estimator $\theta^{(T)}$.  In other words, 
\citet{sher:ho:dala:1999} found conditions under which the MCEM 
approximation to the MLE 
enjoys the same asymptotic properties as 
the MLE itself.  This represents yet another possible notion of 
``convergence'' of MCEM, though not one that we will pursue any further 
in the present paper.  

\citet{fort:moul:2003a} treat the data as fixed, the Monte Carlo sample 
size as increasing (deterministically) across MCEM iterations, and 
establish a$.$s$.$ convergence of the sequence as the iteration count 
goes to infinity.  We consider this the strongest known result on 
the asymptotic properties of MCEM, as this notion of convergence seems 
the most consistent with that of ordinary (deterministic) EM.  We summarize 
\citet{fort:moul:2003a} main conclusions in subsection \ref{sub:fort.moul}.

\subsection{A result of \citet{chan:ledo:1995}}
\label{sub:chan.ledo}

\citet{chan:ledo:1995} showed that, given a suitable starting value, a 
sequence of parameter values generated by the Monte Carlo EM algorithm will 
get arbitrarily close to a maximizer of the observed likelihood with high 
probability.  Their main result is given as Theorem 
\ref{thm:chan.ledolter} below.  We first establish one more convergence 
property of deterministic EM, also attributable to \citet{chan:ledo:1995}. 
%

Let $M_{EM}:\Theta \rightarrow \Theta$ denote the mapping given by the 
deterministic EM update rule, that is, $M_{EM}(\tilde{\theta}) = 
\arg \max Q(\theta | \tilde{\theta}; y)$.  
\begin{lemma}
\label{lem:chan.ledolter}
\citep[Lemma 1 of][]{chan:ledo:1995}  Suppose $\theta^*$ is a 
local maximizer of the log-likelihood $l(\theta; y)$, 
a continous function of $\theta$, and that 
there exists a neighborhood in which $\theta^*$ is the only stationary 
point.  Then for any neighborhood $\mathcal{N}$ of $\theta^*$,  there 
exists a neigborhood $\mathcal{N}^*$ such that an EM sequence 
$\left\{ \theta^{(t)} : t = 0, 1, 2, \ldots \right\}$ started at any 
$\theta^{(0)} \in \mathcal{N}^*$, satisfies (i) $\theta^{(t)} \in 
\mathcal{N}$ for all $t = 1, 2, \ldots$; and (ii) $\theta^{(t)} 
\rightarrow \theta^*$ as $t \rightarrow \infty$.  
\end{lemma}
\begin{proof} 
%
%
%
Let $\mathcal{N}$ be a neighborhood of $\theta^*$.  There exists a compact, 
connected subneighborhood $\mathcal{N}^* \subset \mathcal{N}$ such that 
(i) $l(\theta; y)$ attains its maximum over $\mathcal{N}^*$ at $\theta^*$, 
(ii) $\mathcal{N}^*$ contains no other stationary points of $l$, and (ii) 
there exists $\varepsilon > 0$ such that $l(\theta; y) \geq 
l(\theta^*; y) - \varepsilon$ for all $\theta \in \mathcal{N}^*$.  
It follows from these conditions that 
$M_{EM}(\theta) \in \mathcal{N}^*$ for any $\theta \in \mathcal{N}^*$; thus
an EM sequence $\left\{ \theta^{(t)} \right\}$ with $\theta^{(0)} \in 
\mathcal{N}^*$ satisfies $\theta^{(t)} \in \mathcal{N}^*$, and thus 
$\theta^{(t)} \in \mathcal{N}$, for all $t = 1, 2, \ldots$.  

Continue to assume that 
$\theta^{(0)} \in \mathcal{N}^*$ and consider the EM sequence 
$\left\{ \theta^{(t)} \right\}$.  Now the 
%
%
%
%
%
%
sequence $\left\{ l( \theta^{(t)}; y) \right\}$ is nondecreasing and 
bounded above by $l(\theta^*; y)$, and thus converges to a finite limit, 
call it $\lambda$.  The 
sequence $\left\{ \theta^{(t)} \right\}$ lives in $\mathcal{N}^*$, a compact 
set; let $\left\{\theta^{(t_k)} \right\}$ be a convergent subsequence and 
denote its limit by $\theta^{**} \in \mathcal{N}^*$.  

Suppose $\theta^{**} \neq \theta^*$.  Then $l(\theta^{(t_k+1)}; y) 
\rightarrow l(M_{EM}(\theta^{**}); y) > l(\theta^{**}; y) = \lambda$.  That is, 
the subsequence $\left\{l(\theta^{(t_k+1)}; y) \right\}$ converges to a 
limit greater than $\lambda$, a contradiction.  

Thus any convergent subsequence of $\left\{ \theta^{(t)} \right\}$ must 
converge to $\theta^*$; thus $\left\{\theta^{(t)} \right\}$ converges to 
$\theta^*$.  
\end{proof}
In the terminology of the stability theory of dynamical systems 
\citep[see, for example,][section 3.5]{arro:plac:1992}, the lemma asserts 
that 
an isolated local maximizer $\theta^*$ of $l(\theta; y)$ 
is an \emph{asymptotically stable fixed point} for the EM 
algorithm.  Practically, Lemma \ref{lem:chan.ledolter} tells us that 
an EM sequence with a sufficiently close starting value will remain 
arbitrarily close to $\theta^*$ (by stability) as well as converge to 
$\theta^*$. 
\begin{theorem}
\label{thm:chan.ledolter}
\citep[Theorem 1 of][]{chan:ledo:1995}. Let $\left\{ \theta^{(t)} \right\}$ 
denote a Monte Carlo EM sequence based on Monte Carlo sample sizes 
$m_t \equiv m$, and suppose that the MCEM update $\mathcal{M}_m(
\tilde{\theta}) := \arg \max Q_m(\theta | \tilde{\theta}; y)$ converges in 
probability to $M_{EM}(\tilde{\theta})$ as $m \rightarrow \infty$.  Further 
suppose that this convergence is uniform on compact subsets of $\Theta$.  
Let $\theta^*$ be an isolated local maximizer of $l(\theta; y)$, a 
continous function of $\theta$.  Then there exists a neighborhood of 
$\theta^*$ such that for any starting value $\theta^{(0)}$ in that 
neighborhood and for any $\varepsilon > 0$, there exists $T_0$ such that 
\begin{equation}
\label{eqn:chan.ledo.thm1}
\mathrm{Pr} \left\{ || \theta^{(t)} - 
\theta^* || < \varepsilon ~\mathrm{for}~\mathrm{some}~ t \leq T_0 
\right\}~ \rightarrow~ 1
\end{equation}
as the Monte Carlo sample size $m \rightarrow 
\infty$.  
\end{theorem}
\begin{proof}
Let 
$\mathcal{N}$ be the set 
defined as $\mathcal{N}^*$ in the proof of Lemma \ref{lem:chan.ledolter}, 
so that $\mathcal{N}$ is compact and connected, contains $\theta^*$, and 
$M_{EM}(\theta) \in \mathcal{N}$ for any $\theta \in \mathcal{N}$.  For any 
$\varepsilon > 0$, we will find $T_0$ such that \eqref{eqn:chan.ledo.thm1} 
holds for any $\theta^{(0)} \in \mathcal{N}$.  

Let $\varepsilon > 0$ be given.  First, there exists a positive number 
$\varepsilon_1 \leq \varepsilon$ such that $\mathcal{N}_1 := 
\left\{ \theta \in \mathcal{N}: || \theta - \theta^* || \geq 
\varepsilon_1 \right\}$ is nonempty; 
if $\theta \in \mathcal{N}_1$, then $M_{EM}(\theta) \neq 
\theta$.  By the ascent property and by continuity of $l$ in $\theta$ there 
exist $\delta, ~\delta_1 > 0$ such that for any $\theta \in \mathcal{N}_1$, 
if $||\theta' - M_{EM}(\theta) || < \delta$, then 
$l(\theta'; y) - l(\theta; y) > \delta_1$.  

By construction 
of $\mathcal{N}$, there exists $\delta_2 > 0$ such that for any 
$\theta \in \mathcal{N}$, any $\theta'$ with $|| \theta' - 
M_{EM}(\theta) || < \delta_2$ is also in $\mathcal{N}$.  Without loss 
of generality we can take $\delta_2 < \delta$.  Thus we have that for any 
$\theta \in \mathcal{N}_1$, any $\theta'$ with $|| \theta' - 
M_{EM}(\theta) || < \delta_2$ is also in $\mathcal{N}$ (though not 
necessarily in $\mathcal{N}_1$) and $l(\theta'; y) - l(\theta; y) 
> \delta_1$.  Let 
\begin{equation}
\label{eqn:R.def}
R = \sup_{\theta, \theta' \in \mathcal{N}} \left\{ l(\theta; y) - 
l(\theta'; y) \right\} < \infty
\end{equation}
and let $T_0 =\lfloor R/\delta_1 \rfloor + 1$, 
where $\lfloor \cdot \rfloor$ denotes the greatest integer function.  

Now, suppose an element of the MCEM sequence 
$\theta^{(t)} = \tilde{\theta} \in \mathcal{N}$.  
Then the probability that its MCEM update $\theta^{(t+1)} 
= \mathcal{M}_m(\theta^{(t)})$ is also in $\mathcal{N}$ is
\begin{equation}
\label{eqn:chan.ledo.prf1}
\mathrm{Pr} \left\{ \theta^{(t+1)} \in \mathcal{N} ~\big{\vert}~ 
\theta^{(t)} = \tilde{\theta} \right\} \geq \mathrm{Pr} \left\{ 
|| \theta^{(t+1)} - M_{EM}(\theta^{(t)}) || < \delta_2 ~\big{\vert}~ 
\theta^{(t)} = \tilde{\theta} \right\}
\end{equation}
by the definition of $\delta_2$.  Denote a lower bound on the right hand 
side of \eqref{eqn:chan.ledo.prf1} by $p = p(\delta_2, m) > 0$ and note 
that (i) $p$ can be chosen not to depend on the value of $\tilde{\theta} 
\in \mathcal{N}$ by the compactness of $\mathcal{N}$ and the uniformity of 
convergence $\mathcal{M}_m(\theta) \rightarrow M_{EM}(\theta)$ over compact 
subsets of $\Theta$; and (ii) for fixed $\delta_2$, 
$p(\delta_2, m) \rightarrow 1$ as $m \rightarrow \infty$.  

Consider running a Monte Carlo EM algorithm for $T_0$ updates.  
For any starting value $\theta^{(0)} \in \mathcal{N}$, 
\begin{equation}
\label{eqn:chan.ledo.prf2}
\begin{split}
\mathrm{Pr} \left\{ \theta^{(t)} \in \mathcal{N} ~ \mathrm{for}~ t = 
 \right. & \left. 0, 1, 
\ldots, T_0 \right\} \geq \\ & \mathrm{Pr} \left\{ || \theta^{(t+1)} - 
M_{EM}(\theta^{(t)}) || < \delta_2 ~\mathrm{for}~ t = 0, 1, \ldots, T_0 - 1 
\right\},
\end{split}
\end{equation}
and since each Monte Carlo EM update is calculated independently, the right 
hand side of \eqref{eqn:chan.ledo.prf2} is bounded below by 
$p(\delta_2, m)^{T_0}$.  

Now, suppose that $\theta^{(0)} \in \mathcal{N}$, and that 
$|| \theta^{(t+1)} - M_{EM}(\theta^{(t)}) || < \delta_2$ for each $t$, and thus 
$\theta^{(t)} \in \mathcal{N}$ for each $t$.  Suppose to get a contradiction 
that $|| \theta^{(t)} - \theta^* || \geq \varepsilon_1$, that is, that 
$\theta^{(t)} \in \mathcal{N}_1$ for each $t = 0, 1, \ldots, T_0$.  Then 
$l(\theta^{(t+1)}; y) - l(\theta^{(t)}; y) > \delta_1$ for each $t = 0, 1, 
\ldots, T_0 - 1$, and thus $l(\theta^{(T_0)}; y) - l(\theta^{(0)}; y) > 
\delta_1 T_0 > R$.  But that contradicts \eqref{eqn:R.def}, the definition 
of $R$, since $\theta^{(0)}$ and $\theta^{(T_0)}$ are both in $\mathcal{N}$.  

Thus it must be that if $\theta^{(0)} \in \mathcal{N}$ and 
$|| \theta^{(t+1)} - M_{EM}( \theta^{(t)} ) || < \delta_2$ for each $t$, then 
$|| \theta^{(t)} - \theta^* || < \varepsilon_1 \leq \varepsilon$ for some $t$, 
which occurs with probability not less than $p(\delta_2, m)^{T_0}$, which 
converges to 1 as $m \rightarrow \infty$.  
\end{proof}

A couple of remarks are in order.  First, we note that the assumptions of 
Theorem \ref{thm:chan.ledolter} are slightly different than those made by 
\citet{chan:ledo:1995} in that where we assumed uniform convergence of the 
Monte Carlo EM update, \citet{chan:ledo:1995} assumed conditions on the 
form of the log-likelihood sufficient to guarantee it.  
Secondly, the conclusion of Theorem 
\ref{thm:chan.ledolter}, while interesting, is unsatisfying in at least 
one respect: It does not guarantee the convergence of an MCEM sequence in 
any meaningful sense.  Practically, what this theorem tells us is that if 
you run the algorithm long enough (at least $T_0$ iterations), the 
resulting sequence will, with high probability, \emph{at some point} get 
arbitrarily close to the MLE.  But to an analyst examining the output of 
an MCEM run, even a very long one, there is no way to know when that has 
happened, if at all.  A more powerful result would be one that 
specifies conditions under which the algorithm gets close to the MLE and 
stays there.  

\subsection{A result of \citet{fort:moul:2003a}}
\label{sub:fort.moul}

\cite{fort:moul:2003a} used the ergodic theory of Markov chains to prove 
the almost sure (a.s.) convergence of a variation of the Monte Carlo EM 
algorithm.  We will state their assumptions and main conclusion; the 
proof is highly technical and beyond the scope of this report.

We will state \citet{fort:moul:2003a} convergence result assuming that 
the Monte Carlo E-step is accomplished by i$.$i$.$d$.$ sampling.  In 
fact the result holds more generally under Markov chain Monte Carlo 
methods, assuming the underlying Markov transition kernel is 
\emph{uniformly ergodic} \citep[see, for example,][]{jone:hobe:2001}.  

\citet{fort:moul:2003a} consider a variation of Monte Carlo EM they call 
\emph{stable MCEM}, which we define here.  Let $\left\{ \mathcal{K}_t : 
t = 0, 1, 2, \ldots \right\}$ be a sequence of compact subsets of 
$\Theta$ satisfying
\begin{equation}
\label{eqn:stable.mcem}
\mathcal{K}_t \subset \mathcal{K}_{t+1} ~ ~\mathrm{for}~\mathrm{each}~ t, 
~~ \mathrm{and}~~ \bigcup_{t = 0}^{\infty} \mathcal{K}_t = \Theta \; .
\end{equation}
Set $p_0 = 0$ and choose $\theta^{(0)} \in \mathcal{K}_0$.  Given 
$\theta^{(t)}$ and $p_t$, the stable MCEM update rule for $\theta^{(t+1)}$ 
and $p_{t+1}$ is given by
\begin{enumerate}
\item Let $\theta'$ be the ordinary MCEM update as defined in Section 
\ref{sec:Intro}.
\item If $\theta' \in \mathcal{K}_{p_t}$, then $\theta^{(t+1)} = \theta'$ 
and $p_{t+1} = p_t$.  

If $\theta' \notin \mathcal{K}_{p_t}$, then $\theta^{(t+1)} = \theta^{(0)}$ 
and $p_{t+1} = p_t + 1$.
\end{enumerate}
Thus in stable MCEM, any time the ordinary MCEM update falls outside a 
specific set, the algorithm is reinitialized at the point $\theta^{(0)}$; 
$p_t$ counts the cumulative number of reinitializations as of update $t$.  
\citet{fort:moul:2003a} showed that under appropriate assumptions 
(see Theorem \ref{thm:fort.moul} below), $\left\{ p_t \right\}$ is 
a$.$s$.$ finite.  

We will assume that the 
complete data model $f(y,u; \theta)$ is from the class of 
\emph{curved exponential families}: Let $\mathcal{Y} \subset \mathbb{R}^N$ 
denote the range of $Y$ and $\mathcal{U} \subset \mathbb{R}^q$ the range 
of $U$.  We assume that for some integer $k$ there exist functions $\phi: 
\Theta \rightarrow \mathbb{R}^1$, $\psi : \Theta \rightarrow \mathbb{R}^k$, 
and $S : \mathcal{Y} \times \mathcal{U} \rightarrow \mathcal{S} \subset 
\mathbb{R}^k$ such that 
$$
l_c(\theta; y, u) = \log f(y, u; \theta) = \psi(\theta)^T S(y,u) + 
\phi(\theta) \; .
$$
Since $l_c$ depends on $(y,u)$ only through $s = S(y,u)$ we can write 
$l_c(\theta; s) = \psi(\theta)^T s + \phi(\theta)$.  
Note that the curved exponential families include the linear mixed model 
of Example 1 in Section \ref{sec:2shm}, but not the logit-normal GLMM 
of Example 2.  

We will further assume that
\begin{enumerate}
\item $\phi$ and $\psi$ are continuous on $\Theta$, $S$ is continuous on 
$\mathcal{Y} \times \mathcal{U}$;  
\item for all $\theta \in \Theta$, $\bar{S}(\theta; y) := \mathrm{E} 
\left\{ S(y, U) ~|~  y; ~\theta \right\}$ is finite and continuous on 
$\Theta$;
\item there exists a continuous function $\hat{\theta}: \mathcal{S} 
\rightarrow \Theta$ such that for all $s \in \mathcal{S}$, 
$l_c(\hat{\theta}(s); s) = \sup_{\theta \in \Theta} l_c(\theta; s)$; 
\item the observed data log-likelihood $l(\theta; y)$ is continuous on 
$\Theta$, and for any $\lambda$, the level set $\left\{ \theta \in \Theta : 
l(\theta; y) \geq \lambda \right\}$ is compact;
\item the set of fixed points of the EM algorithm is compact.
\end{enumerate}
Let $\Gamma$ denote the set of fixed points of the EM algorithm; in a 
curved exponential family, and using the notation introduced above, 
$\Gamma = \left\{ \theta \in \Theta : \hat{\theta}( \bar{S}(\theta; y) ) = 
\theta \right\}$. 
As shown by \citet[][Theorem 2]{wu:1983}, under the above assumptions, 
if $\Theta$ is open and $\phi$ and $\psi$ are differentiable on $\Theta$, 
then $l(\theta; y)$ is differentiable on $\Theta$ and $\Gamma = 
\left\{ \theta \in \Theta : \nabla l(\theta; y) = 0 \right\}$.  In 
other words, the set of fixed points of the EM algorithm coincides with 
the set of stationary points of the log-likelihood $l(\theta; y)$; see 
also our Theorem \ref{thm:em.thm2}.  

%
Finally, note that assumptions $4.$ and $5.$ guarantee that the set 
$\left\{ \l(\theta; y) : \theta \in \Gamma \right\}$
is compact as well.  
We can now state \pcite{fort:moul:2003a} main result.  We will denote the 
\emph{closure} of a sequence by $\mathrm{Cl}(\cdot)$, so that 
$\mathrm{Cl}( \left\{ \theta^{(t)} \right\} )$ represents the union of the 
sequence $\left\{ \theta^{(t)} \right\}$ itself with its limit points.

\begin{theorem}
\label{thm:fort.moul} \citep[Theorem 3 of][]{fort:moul:2003a} Assume the 
complete data model is from the class of curved exponential families, and 
the model satisfies assumptions $1.$ through $6.$ above.  Consider an 
implementation of the stable MCEM algorithm using a sequence of sets $\left\{ 
\mathcal{K}_t \right\}$ satisfying \eqref{eqn:stable.mcem}.  Let 
$\theta^{(0)} \in \mathcal{K}_0$ and suppose the Monte Carlo sample sizes 
$\left\{ m_t \right\}$ satisfy $\sum_{t=0}^{\infty} m_t^{-1} < \infty$.  Then
\begin{enumerate}
\item 
\begin{enumerate}
\item $\lim_{t \rightarrow \infty} p_t < \infty$ with probability 1 (w$.$p$.$ 1) 
and $\limsup_{t \rightarrow \infty} || \theta^{(t)} || < \infty$ w$.$p$.$ 1;
\item $\left\{ l(\theta^{(t)}; y) \right\}$ converges w$.$p$.$ 1 to a 
connected component of $l(\Gamma; y)$ where $\Gamma$ denotes the set of 
stationary points of $l(\theta; y)$ (and fixed points of the EM algorithm).
\end{enumerate}
\item If, in addition, $\left\{ l( \theta; y ) : \theta \in 
\Gamma \cap \mathrm{Cl}( \left\{ \theta^{(t)} \right\} ) \right\}$ has an empty 
interior, then $\left\{ l(\theta^{(t)}; y) \right\}$ converges w$.$p$.$ 1 to 
a point $\lambda^*$ and $\left\{ \theta^{(t)} \right\}$ converges to the set 
$\left\{ \theta : l(\theta; y) = \lambda^* \right\}$.  
\end{enumerate}
\end{theorem}
It is often the case that the set $\Gamma$ is made up of isolated points; 
the above theorem then guarantees pointwise convergence of 
$\left\{ l(\theta^{(t)}; y) \right\}$ to a stationary point of $l(\theta; y)$.  
If $\Gamma$ consists of a single point $\hat{\theta}$, the theorem guarantees 
that $l(\theta^{(t)}; y) \rightarrow l(\hat{\theta}; y)$ w$.$p$.$ 1 
and $\theta^{(t)} \rightarrow \hat{\theta}$ w$.$p$.$ 1, analogous to 
Corollary \ref{cor:unique.mle}.  

Finally, we note that the assumption that $\sum m_t^{-1} < \infty$ can be 
weakened in many instances, but is necessarily of the form 
$\sum m_t^{-p} < \infty$ for some $p \geq 1$.



\section{Remarks: Lessons for the (MC)EM practitioner}
\label{sec:conclusion}


We conclude with a brief discussion of the practical implications of 
the convergence results of Sections \ref{sec:EM} and  \ref{sec:MCEM}.  
First, as we noted in our discussion following Theorem \ref{thm:em.thm2}, 
even when EM converges, there is no guarantee in general that it has 
converged to a global maximum.  In more complex settings such as mixture 
models, or model-based clustering, the likelihood function may have 
multiple optima, most of which will be local optima.  While the EM 
algorithm may converge, its limit point is sub-optimal.  Solutions to 
overcome local optima can include merging the ideas of the EM algorithm 
with those of global optimization.  One example is described in the paper 
by \citet{heat:fu:jank:2009} who combine EM with the cross-entropy 
method and model reference adaptive search, two global optimization 
heuristics.  Another example can be found in \citet{tu:ball:jank:2008}, 
who combine the EM algorithm with the genetic algorithm to model flight 
delay distributions.  

With respect to Monte Carlo EM, 
as we have
%
%
previously noted, the Monte Carlo sample size must be increased with the 
iteration count; otherwise there is no chance for convergence in the 
usual sense, due to the persistence of Monte Carlo error.  The 
convergence results of section 4.2 \citep{fort:moul:2003a} require 
$\sum m_t^{-1} < \infty$.  Intuitively it makes sense to start the 
algorithm with modest simulation sizes: when the parameter value is 
relatively far from the MLE, the (deterministic) EM update makes a 
substantial jump, and less precision is required for the Monte Carlo 
approximation to that jump.  When the parameter value is close to the 
MLE, as will be the case after a number of iterations, the EM update 
is a small step, and greater precision is required for the Monte Carlo 
approximation.  

Thus it is clear that $m_t$ must be an increasing function of $t$, though 
it is not at all clear what might be an appropriate form.  In fact there 
exists a literature, beginning with \citet{boot:hobe:1999}, on 
\emph{automated} Monte Carlo EM algorithms, in which the simulation size 
for each Monte Carlo E-step is determined internally to the algorithm, 
based on some rule for assessing the level of precision required for the 
Monte Carlo approximation at hand.  Other authors who have contributed 
to this literature include \citet{levi:case:2001} and 
\citet{caff:jank:jone:2005}.  

One can view 
the Monte Carlo EM update to the parameter value $\theta^{(t)}$ as an 
estimate of the deterministic EM update $M_{EM}(\theta^{(t)})$.  
In \pcite{boot:hobe:1999} algorithm, each MCEM update requires the 
computation of an asymptotic confidence region for $M_{EM}(\theta^{(t)})$
in addition to the point estimate $\mathcal{M}_{m_t}(\theta^{(t)})$.  
If $\theta^{(t)}$ falls within this confidence region, we must accept 
that the current parameter value $\theta^{(t)}$ is statistically 
indistinguishable from its EM update $M_{EM}(\theta^{(t)})$.  This suggests 
that the MCEM update was ``swamped by Monte Carlo error,'' and thus 
the simulation 
size must be increased at the next iteration.  
The reader is referred to \citet{boot:hobe:1999} for details and 
examples.  \citet{levi:case:2001} use a regeneration-based approach to 
Monte Carlo standard errors in computing their confidence region.  

The Ascent-based Monte Carlo EM algorithm of \citet{caff:jank:jone:2005} 
seeks to prevent the MCEM update from being swamped by Monte Carlo error 
by successively appending the Monte Carlo sample until one has a 
pre-specified level of confidence that the proposed update increases 
the log-likelihood over the current parameter value, that is, until we 
are confident that indeed $\l(\theta^{(t+1)}; y) \geq l(\theta^{(t)}; y)$.  
Recall that this 
ascent property 
is guaranteed for ordinary EM (Theorem \ref{thm:ascent}).  Since the 
MCEM update maximizes an estimate of the $Q$-function rather than the 
$Q$-function itself, there is no ascent property for MCEM in general.  
But a parameter update computed according to the Ascent-based MCEM rule 
will increase the log-likelihood with high probability.  Again the reader 
is referred to the source \citep{caff:jank:jone:2005} for details.  
Empirical comparisons between Ascent-based MCEM and 
\pcite{boot:hobe:1999} algorithm can be found in 
\citet{caff:jank:jone:2005} and \citet{neat:2006}.  

A second practical implication of the convergence properties of Monte 
Carlo EM relates to convergence criteria, or \emph{stopping rules} for 
the algorithm.  At what point should the MCEM iterations be terminated 
and the current parameter value accepted as the MLE?  The usual stopping 
rules employed in a deterministic iterative algorithm like ordinary EM 
terminate when it is apparent that further iterations (i) will not 
substantively change the approximation to the MLE, or (ii) will not 
substantively change the value of the objective (likelihood) function.  
For example, one might terminate at the first iteration $t$ to satisfy
\begin{equation}
\label{eqn:stopping}
\max_i \left\{ \frac{ | \theta_i^{(t)} - \theta_i^{(t-1)} | }{ 
| \theta_i^{(t)} | + \delta } \right\} < \epsilon
\end{equation}
for user-specified $\delta$ and $\epsilon$, where the maximum is taken 
over components of the parameter vector.  In Monte Carlo EM, such criteria 
run the risk of terminating too early, as \eqref{eqn:stopping} may be 
attained only because of Monte Carlo error in the update.  An obvious but 
inelegant solution is to terminate only after \eqref{eqn:stopping} is 
met for, say, three consecutive iterations.  This is the stopping rule 
recommended by \citet{boot:hobe:1999}.  Other MCEM stopping rules 
considered in the literature include \pcite{chan:ledo:1995} suggestion 
to terminate at the first iteration where $l(\theta^{(t)}; y) - 
l(\theta^{(t-1)}; y)$ is stochastically small; in a similar vein 
\citet{caff:jank:jone:2005} terminate when an asymptotic upper bound 
on $Q(\theta^{(t)} | \theta^{(t-1)}; y) - Q(\theta^{(t-1)} | \theta^{(t-1)}; y)$ 
falls below a pre-specified tolerance.  

Finally, we note that while our focus throughout has been on finding a 
good approximation to the MLE, meaningful statistical inference requires 
at minimum a reliable estimate of the standard error as well.  
A formula in \citet{loui:1982} expresses the observed Fisher Information 
as an expectation taken with respect to the conditional distribution of 
the unobserved data given the observed data.  Thus a Monte Carlo 
approximation to the inverse covariance matrix of the MLE is readily 
available from the simulation already conducted to compute the final 
MCEM update.

\bibliographystyle{imsart-nameyear}
\bibliography{mcref}

\end{document}